\documentclass[11pt,leqno]{article}

\usepackage[T1]{fontenc}
\usepackage{amssymb,graphics,epsfig}
 \usepackage[latin1]{inputenc}
          \usepackage{latexsym,amsmath,amsthm,amscd}
\usepackage{amsfonts}

\def\d{\mbox{d}}
\def\ep{\varepsilon}
\def\p{\partial}
\def\Om{\Omega}

\def\R{\mathbb{R}}
\def\N{\mathbb{N}}

\def\divgce{\mbox{div}}

\newtheorem{theo}{Theorem}[section]
\newtheorem{prop}[theo]{Proposition}
\newtheorem{lemma}[theo]{Lemma}

\title{What is the optimal shape of a pipe?}
\author{Antoine Henrot\footnote{corresponding author}\\ Institut \'Elie Cartan, UMR 7502,
Nancy Universit\'e - CNRS - INRIA\\
B.P. 239 54506 Vandoeuvre les Nancy Cedex,  France\\
email: henrot@iecn.u-nancy.fr \and Yannick Privat\\ Institut \'Elie
Cartan, UMR 7502,
Nancy Universit\'e - CNRS - INRIA\\
B.P. 239 54506 Vandoeuvre les Nancy Cedex,  France\\
email: Yannick.Privat@iecn.u-nancy.fr}

\begin{document}

\maketitle

\vspace*{0.4cm} {\bf Abstract.} We consider an incompressible fluid
in a three-dimensional pipe, following the Navier-Stokes system with
classical boundary conditions. We are interested in the following
question: is there any optimal shape for the criterion "energy
dissipated by the fluid"? Moreover, is the cylinder the optimal
shape? We prove that there exists an optimal shape in a reasonable
class of admissible domains, but the cylinder is not optimal. For
that purpose, we explicit the first order optimality condition,
thanks to adjoint state and we prove that it is impossible that the
adjoint state be a solution of this over-determined system when the
domain is the cylinder. At last, we show some numerical simulations
for that problem. \vspace*{0.3cm}

{\bf Keywords:} \parbox[t]{11cm}{shape optimization, Navier-Stokes, symmetry}\\

{\bf AMS classification:} primary: 49Q10, secondary: 49J20, 49K20, 35Q30, 76D05, 76D55\\

\section{Introduction}
The shape optimization problems in fluid mechanics are very
important and gave rise to many works. Most often, these works have
a numerical character due to the intrinsic difficulty of the
Navier-Stokes equations. For a first bibliography on the topic, we
refer e.g. to \cite{fei}, \cite{HP}, \cite{mo-pi}, \cite{Pi}
\cite{PlSo}.

In this work, we are interested in one of the simplest problem: what
shape must have a pipe in order to minimize the energy  dissipated
by a fluid? For us, a pipe (of "length" $L$) will be a three
dimensional domain $\Om$ contained in the strip $\left\{
(x_1,x_2,x_3) \;, 0< x_3< L\right\}$. We will assume that the inlet
$E:=\p\Om \cap \{x_3=0\}$ (where $\p\Om$ denotes the boundary of
$\Om$) and the outlet $S:=\p\Om \cap \{x_3=L\}$ are two fixed
identical discs and that the volume of $\Om$ is imposed. The unknown
(or free) part of the boundary of $\Omega$ will be denoted by
$\Gamma$ (so $\p\Om=E\cup\Gamma\cup S$).

In the pipe $\Om$, we consider the flow of a viscous, incompressible
fluid with a velocity $\mathbf{u}$ and a pressure $p$ satisfying the
Navier-Stokes system. We assume that the velocity profile
$\mathbf{u_0}$ at the inlet $E$ is of parabolic type; on the lateral
boundary $\Gamma$, we assume no-slip condition $\mathbf{u}=0$ and we
control the outlet by imposing an "outlet-pressure" condition on
$S$. We will assume that the viscosity $\mu$ is large enough in
order that the solution of the system is unique (see \cite{temam}).
The criterion that we want to minimize, with respect to the shape
$\Om$, is the energy dissipated by the fluid (or viscosity energy)
defined by $ J(\Omega ):=2\mu \int_\Omega |\varepsilon
(\mathbf{u})|^2\d x$ where $\varepsilon$ is the stretching tensor.

We will first prove an existence Theorem. To obtain this result, we
work in the class of admissible domains which satisfy an $\ep$-cone
property (see \cite{Che1}, \cite{HP}). Then, we are interested in
symmetry properties of the optimal domain. For the Stokes model, we
are only able to prove that the optimum has one plane of symmetry.
It is not completely clear to see whether the optimum should be
axially symmetric. In a series of papers \cite{AP}, \cite{PA}, G.
Arumugam and O. Pironneau  proved for a similar, but much simpler
problem that one has to build riblets on the lateral boundary to
reduce the drag. Nevertheless, it is a natural question to ask
whether the cylinder should be the optimum for our problem. We will
show that it is not the case. For that purpose, we explicit the
first order optimality condition. This condition can be easily
expressed in term of the adjoint state and gives an over-determined
condition on the lateral boundary $\Gamma$. Then, we prove that it
is impossible that the adjoint state be a solution of this
over-determined system when the domain is the cylinder.

This paper is organized as follows. At section \ref{sec2}, we state
the shape optimization problem, we prove existence and symmetry.
Section \ref{sec3} is devoted to the proof of the main Theorem. We
give in section \ref{sec4} some numerical results and concluding
remarks.

These results have been announced in the Note \cite{HPri}.

\section{The shape optimization problem}\label{sec2}
Let us give the notations used in this paper. We consider a generic
three dimensional domain $\Om$ contained in a compact set
$$D:=\left\{ (x_1,x_2,x_3) \;, x_1^2+x_2^2\leq R_0^2 \;,0\leq x_3\leq
L\right\}$$
where $R_0$ and $L$ are two positive constants. We will
denote by $\p\Om$ the boundary of $\Om$. In the sequel, we will
assume that the inlet $E$ of $\Om$ defined by $E:=\p\Om \cap
\{x_3=0\}$  and the outlet $S$ defined by $S:=\p\Om \cap \{x_3=L\}$
are two fixed identical discs of radius $R<R_0$ centered on the
$x_3$ axis. We will also assume that the volume of all the domains
$\Om$ is imposed, say $|\Om|=V=\pi R^2L$. We decompose the boundary
of $\Om$ as the disjoint union $\p\Om=E\cup\Gamma\cup S$ and
$\Gamma$, the lateral boundary is the main unknown or the shape we
want to design.

Let us now precise the state equation. We consider the flow of a
viscous incompressible fluid into $\Om$. We denote by
$\mathbf{u}=(u_1,u_2,u_3)$ (letters in bold will correspond to
vectors) its velocity and by $p$ its pressure. As usual in fluid
mechanics, we introduce $\varepsilon$ the stretching tensor defined
by:
$$
\varepsilon (\mathbf{u})=\left(\frac{1}{2}\left(\frac{\partial
u_i}{\partial x_j}+\frac{\partial u_j}{\partial
x_i}\right)\right)_{1\leq i,j \leq 3}.
$$
We will consider the Navier-Stokes system (except for Theorem
\ref{theosym} where the Stokes system will be considered). As
boundary conditions, we assume that the velocity profile
$\mathbf{u_0}$ at the inlet $E=\{x_3=0\}$ is of parabolic type; on
the lateral boundary $\Gamma$, we assume adherence or no-slip
condition $\mathbf{u}=0$ and we control the outlet by imposing an
"outlet-pressure" condition on $S=\{x_3=L\}$. Therefore, the p.d.e.
system satisfied by the velocity and the pressure is:
\begin{equation}
\label{NSinit} \left\{\begin{array}{ll} \displaystyle -\mu \triangle
\mathbf{u}+\nabla p+\nabla
\mathbf{u}\cdot \mathbf{u}=0 & \mathbf{x}\in \Omega , \\
\displaystyle \divgce\, \mathbf{u}=0 & \mathbf{x}\in \Omega \\
\mathbf{u}=\mathbf{u}_0:=\left(0,0,c(x_1^2+x_2^2-R^2)\right) & \mathbf{x}\in E \\
\mathbf{u}=0 & \mathbf{x}\in \Gamma \\
-p\mathbf{n}+2\mu \varepsilon (\mathbf{u})\cdot
\mathbf{n}=\mathbf{h}:=\left(2\mu cx_1,2\mu c x_2,-p_1\right) &
\mathbf{x}\in S.
\end{array}
\right.
\end{equation}
where $\mu>0$ denotes the viscosity of the fluid, $\mathbf{n}$ the
exterior unit normal vector (on $S$ we have $\mathbf{n}=(0,0,1)$).
At last, the constant $c$ which appears in the boundary condition on
$E$ and $S$ is assumed to be negative. The sign of $c$ can
physically be explained. Indeed, in the case where $\Om$ is a
cylinder, the flow is driven by a Poiseuille law (simplified
physical law derived from the Navier-Stokes system which describes a
slow viscous incompressible flow through a constant circular
section). Then , this constant $c$ can be written $c=\displaystyle
\frac{p_1-p_0}{4\mu L}$, where $p_1$ denotes the constant value of
the pressure at the outlet $S$ while $p_0$ is the constant value of
the pressure at the inlet $E$.

This choice of the boundary condition ensures that the solution of
(\ref{NSinit}) will be given by a parabolic profile when $\Om$ is a
cylinder. More precisely, if $\Om$ is the cylinder of radius $R$ and
height $L$, the solution of (\ref{NSinit}) is explicitly given by:
\begin{equation}\label{expli}
\left\{\begin{array}{c}
         \mathbf{u}(x_1,x_2,x_3)=\left(0,0,c(x_1^2+x_2^2-R^2)\right) \\
         p(x_1,x_2,x_3)=4\mu c (x_3-L)+p_1\,.
       \end{array}\right.
\end{equation}
More generally, if $\Om$ is a regular domain, we have a classical
existence and uniqueness result for such systems, see e.g.
\cite{boyer}, \cite{temam}.
\begin{theo}
Let us assume that $\mathbf{u_0}$ belongs to the Sobolev space
$(H^{3/2}(E))^3$ and $\mathbf{h}\in (H^{1/2}(S))^3$. If the
viscosity $\mu$ is large enough, the problem (\ref{NSinit}) has a
unique solution $(\mathbf{u},p)\in H^1(\Omega )\times L^2(\Omega )$.
\end{theo}

The criterion we want to minimize is the energy dissipated by the
fluid (or viscosity energy) defined by:
\begin{equation}\label{crit}
    J(\Omega ):=2\mu \int_\Omega |\varepsilon (\mathbf{u})|^2\d x,
\end{equation}
where $\varepsilon$ is the stretching tensor :
$$
\varepsilon (\mathbf{u})=\left(\frac{1}{2}\left(\frac{\partial
u_i}{\partial x_j}+\frac{\partial u_j}{\partial
x_i}\right)\right)_{1\leq i,j \leq 3}\,.
$$
To make the statement precise, we also need to define the class of
admissible domains or shapes. We will consider a first general
class:
\begin{equation}\label{class1}
\begin{array}{c}
  \mathcal{O}_{V} \stackrel{\textnormal{déf}}{=} \left\{\Omega
\textnormal{ bounded and simply connected domain in } \R ^3 : \right.\\
 \left. |\Omega|=V, \ \Pi_0\cap \overline{\Omega}=E, \
\Pi_L\cap \overline{\Omega}=S,\right\}
\end{array}
\end{equation}
where $\Pi_0$ and $\Pi_L$ denote respectively the planes $\{x_3=0\}$
and $\{x_3=L\}$.

To prove an existence result, we need to restrict the class of
admissible domains. It is a very classical feature in shape
optimization, since these problems are often ill-posed, see
\cite{allaire}, \cite{HP}. We adopt here the choice made by D.
Chenais in \cite{Che1} which consists in assuming some kind of
uniform regularity. More precisely, we will consider domains which
satisfy an uniform cone condition, we say that these domains have
the $\ep$-cone property, we refer to \cite{Che1}, \cite{DZ} or
\cite{HP} for the precise definition. So, we define the class
\begin{equation}\label{class}
    \mathcal{O}_{V}^\ep := \left\{\Omega \in
\mathcal{O}_{V} : \Om \textnormal{ has the $\ep$-cone
property}\right\}
\end{equation}
\begin{lemma}\label{OepsFerme}
The class $\mathcal{O}_{V}^\varepsilon$ is closed for the Hausdorff
distance.
\end{lemma}
\begin{proof} We recall that the class of open sets with the
$\ep$-cone property is closed for the Hausdorff convergence (see
Theorem 2.4.10 in \cite{HP}). Moreover, the convergence also holds
for characteristic functions, so the volume constraint is preserved.
So, it remains just to prove that the properties defining the inlet
$E$ and the outlet $S$ are preserved. Let $(\Omega _n)_{n\in \N}$ be
a sequence of domains in $\mathcal{O}_{V}^\varepsilon$ which
converges, for the Hausdorff distance, to a domain $\Om$. We want to
prove that $\Pi_ 0\cap \overline{\Omega}=E$ and $\Pi_ L\cap
\overline{\Omega}=S$. The first inclusion $\Pi_ 0\cap
\overline{\Omega}\subset E$ is just a consequence of the stability
of inclusion for the Hausdorff convergence of compact sets. Let us
prove the reverse inclusion: let $\mathbf{x_0}\in E$ and $n\in \N$.
Since $\Omega _n$ has the $\varepsilon$-cone property, there exists
a unit vector $\xi _n$ such that the cone $C(\varepsilon ,
\mathbf{x_0},\xi _n)$ be contained in $\Omega _n$. Up to a
subsequence, one can assume that $(\xi _n)$ converges to some unit
vector $\xi$ and that the sequence of cones $ C(\varepsilon ,
\mathbf{x_0},\xi _n)$ converges (for the Hausdorff distance) to the
cone $C(\varepsilon , \mathbf{x_0},\xi )$.  By stability with
respect to inclusion, one has
$$
\left. \begin{array}{r}
\forall n \in \N , \ C(\varepsilon , \mathbf{x_0},\xi _n)\subset \Omega_n \\
C(\varepsilon , \mathbf{x_0},\xi _n)\xrightarrow[n\to +\infty]{H}C(\varepsilon , \mathbf{x_0},\xi )\\
\Omega _n \xrightarrow[n\to +\infty]{H}\Omega
\end{array}\right\}\Longrightarrow C(\varepsilon , \mathbf{x_0},\xi )\subset \Omega .
$$
Therefore $\mathbf{x_0}\in \overline{\Omega}$, and since
$\mathbf{x_0}\in E\subset \Pi _0$, the reverse inclusion is proved.
\end{proof}

\medskip
We are now in position to give our existence result.
\begin{theo}
The problem
\begin{equation}\label{PbOptimExist}
\left\{\begin{array}{l}
\min J(\Omega )\\
\Omega \in \mathcal{O}_{V}^\ep,
\end{array}\right.
\end{equation}
where $J$ is defined in (\ref{crit}) with $\mathbf{u}$ the velocity,
solution of the Navier-Stokes problem (\ref{NSinit}), and
$\mathcal{O}_{V}^\ep$ is defined in (\ref{class}), has a solution.
\end{theo}
\begin{proof} Let $(\Omega _n)_{n\in \N}$, be a minimizing sequence in
$\mathcal{O}_{V}^\varepsilon$. Since the open sets $\Omega _n$ are
contained in a fixed compact set $D$, there exists a subsequence,
still denoted by $\Om_n$ which converges (for the Hausdorff
distance, but also for the other usual topologies) to some set
$\Omega$. Moreover, according to Lemma \ref{OepsFerme}, $\Omega$
belongs to the class $\mathcal{O}_{V}^\varepsilon$.

\par To prove the existence result, it remains to prove continuity
(or lower-semi continuity) of the criterion $J$. For any $n\in \N$,
we denote by $\mathbf{u_n}$ and $p_n$ the solution of the
Navier-Stokes system (\ref{NSinit}) on $\Omega _n$. Due to the
homogeneous Dirichlet boundary condition on the lateral boundary
$\Gamma$, we can extend by zero $\mathbf{u_n}$ and $p_n$ outside
$\Omega_n$. So we can consider that the functions are all defined on
the box $D$ and the integrals over $\Om_n$ and over $D$ will be the
same. Let us first remark that $(\mathbf{u_n})$ is uniformly bounded
in $H^1(D)$. Indeed, the sequence $\int_{\Omega_n} |\varepsilon
(\mathbf{u_n})|^2\d x=\int_D |\varepsilon (\mathbf{u_n})|^2\d x$ is
bounded by definition and the result follows using Korn's inequality
on the set $D$ together with a Poincaré's inequality (see below
proof of proposition \ref{ExistenceRegulariteAdjoint}).

\par Therefore, according to reflexivity of $H^1$ and the
Rellich-Kondrachov's Theorem, there
exists a vector $\mathbf{u}\in [H^1(D)]^3$ and a subsequence, still
denoted  $ \mathbf{u_n}$ such that :
$$
\mathbf{u_n}\stackrel{H^1}{\rightharpoonup }\mathbf{u}\textnormal{
and }\mathbf{u_n}\xrightarrow[]{L^q}\mathbf{u}, \ \forall q\in
[1,6[.
$$
It remains to prove that $\mathbf{u}$ is the velocity solution of
the Navier-Stokes system on $\Om$. Let us write the variational
formulation of (\ref{NSinit}). For any function $\mathbf{w}$
satisfying
$$
\mathbf{w}\in [H^1(D )]^3:\mathbf{w}=0\textnormal{ on }E\cup \Gamma
\textnormal{ and }\divgce \mathbf{w}=0 \textnormal{ in }D,
$$
and for all $n\in \N$, the function $\mathbf{u_n}$ verifies :
\begin{equation}\label{vfNS}
    \int_D \left(2\mu \varepsilon (\mathbf{u_n}):\varepsilon
(\mathbf{w})+\nabla \mathbf{u_n}\cdot \mathbf{u_n}\cdot
\mathbf{w}\right)\d x =\int_{S} \mathbf{h}.\mathbf{u_n}\cdot
\mathbf{w}\d s
\end{equation}
Since we have weak convergence of $\mathbf{u_n}$, it comes :
$$
\int_D \varepsilon (\mathbf{u_n}):\varepsilon (\mathbf{w})\d
x\xrightarrow[n\to +\infty]{}\int_D \varepsilon
(\mathbf{u}):\varepsilon (\mathbf{w})\d x.
$$
Let us now have a look to the trilinear term. We already know that
$\nabla \mathbf{u_n}\stackrel{L^2(D)}{\rightharpoonup}\nabla
\mathbf{u}$. Moreover, from Cauchy-Schwarz's inequality and
Sobolev's embedding Theorem, we have:
\begin{eqnarray*}
\Arrowvert (\mathbf{u_n}-\mathbf{u})\cdot \mathbf{w}\Arrowvert _{[L^2(D)]^3}^2
& \leq & \sum_{i=1}^3\sqrt{\int_{\Omega} (u_{n,i}-u_i)^4\d x\int_{\Omega} w_i^4\d x}\\
 & \leq & 3\Arrowvert \mathbf{u_n}-\mathbf{u}
 \Arrowvert _{[L^4(D)]^3}^2\Arrowvert \mathbf{w}\Arrowvert _{[L^4(D)]^3}^2.
\end{eqnarray*}
Then $(\mathbf{u_n}\cdot \mathbf{w})_{n\in \N}$ converges strongly
in $L^2(D)$ to $\mathbf{u}\cdot \mathbf{w}$. Therefore,
$$
\int _D \nabla \mathbf{u_n}\cdot \mathbf{u_n}\cdot \mathbf{w}\d x
\xrightarrow[n\to +\infty]{}\int_D\nabla \mathbf{u}\cdot
\mathbf{u}\cdot \mathbf{w}\d x.
$$
Finally, weak convergence of $\mathbf{u_n}$ in $[H^1(D)]^3$ implies
weak convergence of the trace in $L^2(S)$ and the boundary term
$\int_{S} \mathbf{h}.\mathbf{u_n}\cdot \mathbf{w}\d s$ in
(\ref{vfNS}) converges to $\int_{S} \mathbf{h}.\mathbf{u}\cdot
\mathbf{w}\d s$. Therefore, $\mathbf{u}$ satisfies the variational
formulation (\ref{vfNS}) (and also the boundary condition
$\mathbf{u}=\mathbf{u_0}$ on $E$ because every $\mathbf{u_n}$
satisfies it). To conclude, it remains to prove that $\mathbf{u}$ is
zero on the lateral boundary $\Gamma$. It is actually a consequence
of the convergence in the sense of compacts of $\Omega _n$ to
$\Omega$, and the fact that $\Omega$ is Lipschitz and then stable in
the sense of Keldys. We refer to Theorem 2.4.10 and Theorem 3.4.7 in
\cite{HP}. \end{proof}

\medskip
We are now concerned with symmetry properties of the minimizer. When
the state system is Stokes instead of Navier-Stokes the following
result can be proved:
\begin{theo}\label{theosym}
There exists a minimizer of the problem (\ref{PbOptimExist}) (with
the Stokes system as state equation) which has a plane of symmetry
containing the vertical axis.\\
Moreover, any minimizer of class $C^2$ has such a plane symmetry.
\end{theo}
\begin{proof}
Let $\Om$ denotes (one of) the minimizer(s) of problem
(\ref{PbOptimExist}) and $D$ the vertical axis $x_1=x_2=0$. Among
every plane containing $D$, at least one, say $P_0$, cuts $\Om$ in
two sub-domains $\Om_1$ and $\Om_2$ of same volume (volume equals to
$V/2$).
\par Let us now introduce the two quantities $J_1$ and $J_2$ defined by:
$$
J_1:=2\mu \int_{\Omega _1}|\varepsilon (\mathbf{u})|^2 \d x\
\textnormal{ and }\ J_2:=2\mu \int_{\Omega _2}|\varepsilon
(\mathbf{u}|^2\d x ,
$$
so $J(\Omega )=J_1+J_2$. Without loss of generality, one can assume
$J_1\leq J_2$. Let us now consider the new domain $\widehat{\Omega
}=\Omega _1\cup \sigma (\Omega _1)$, where $\sigma$ denotes the
plane symmetry with respect to $P_0$. We also introduce the
functions $(\mathbf{\widehat{u}},\widehat{p})$ defined by
$$
\mathbf{\widehat{u}}(\mathbf{x})=\left\{\begin{array}{ll}
\mathbf{u }(\mathbf{x}) & \textnormal{if }\mathbf{x}\in \Omega _1 \\
\mathbf{u }(\sigma (\mathbf{x})) & \textnormal{if }\mathbf{x}\in
\sigma (\Omega _1)
\end{array}\right.
\textnormal{ and } \widehat{p}(\mathbf{x})=\left\{\begin{array}{ll}
p (\mathbf{x}) & \textnormal{if }\mathbf{x}\in \Omega _1 \\
p (\sigma (\mathbf{x})) & \textnormal{if }\mathbf{x}\in \sigma
(\Omega _1)
\end{array}\right.
$$
It is clear that $\mathbf{\widehat{u}}\in
[H^1(\widehat{\Omega})]^3$, $\widehat{p}\in L^2(\widehat{\Omega})$
and $\divgce \,\mathbf{\widehat{u}}=0$. Moreover
$$
2\mu \int _{\widehat{\Omega }}|\varepsilon (\mathbf{\widehat{u
}})|^2\d x=4\mu \int _{\Omega _1}|\varepsilon (\mathbf{u ^\star
})|^2\d x=2J_1 \leq J(\Omega ).
$$
Now, it is well known that the solution of our Stokes problem can
also be defined as the unique minimizer of the functional
$$\psi_\Om(\mathbf{v})):=2\mu \int_{\Omega
}|\varepsilon (\mathbf{v})|^2\d x$$ on the space
$$
V(\Omega ):=\{\mathbf{v}\in H^1(\Omega ) : \divgce \mathbf{v}=0, \
\mathbf{v}_{\mid _E}=\mathbf{u_0}\textrm{ and }\mathbf{v}_{\mid
_\Gamma}=0\}.
$$
Therefore, we have:
\begin{equation}\label{sym2}
\begin{array}{rcl}
\displaystyle  J(\widehat{\Omega }) & = & \min _{\mathbf{v}\in
V(\widehat{\Omega })}\left(2\mu
\int_{\widehat{\Omega }}|\varepsilon (\mathbf{v})|^2\d x\right) \\
\displaystyle  & \leq & 2\mu \int_{\widehat{\Omega }}|\varepsilon
 (\mathbf{\widehat{u}})|^2\d x\;\leq  J(\Omega ),
\end{array}
\\
 \\
\end{equation}
this proves that $\widehat{\Omega }$, which has the same volume as
$\Om$ and is symmetric with respect to $P_0$, is also a minimizer of
$J$.

\medskip
Now, let us prove that if $\Om$ is regular enough (actually $C^2$
but one can weaken as shown by the proof below), it must coincide
with $\widehat{\Omega }$, and therefore is symmetric. Necessarily,
we must have the equality in the chain of inequalities (\ref{sym2}).
It proves, in particular, that $\mathbf{\widehat{u}}$ is the
solution of the Stokes problem on $\widehat{\Omega }$. But since
$\mathbf{\widehat{u}}$ coincides with $\mathbf{u}$ on $\Om_1$ by
definition, one can use the analyticity of the solution of the
Stokes problem (see e.g. \cite{morrey}) to claim that
$\mathbf{\widehat{u}}=\mathbf{u}$ on $\Om\cap \widehat{\Omega }$.
Now, if $\widehat{\Omega }$ would not coincide with $\Om$, we would
have a part of the boundary of $\Om$, say $\gamma$ included in
$\widehat{\Omega }$. By assumption, $\Om$ being $C^2$, the solution
of the Stokes problem is continuous up to the boundary (see
\cite{galdi}) and therefore $\mathbf{\widehat{u}}$ should vanish on
$\gamma$. By analyticity, it would imply that it vanishes
identically: a contradiction with the boundary condition on $E$.
\end{proof}

\medskip
As explained in the introduction, one can wonder whether the
minimizer has more symmetry. In particular, could the cylinder be
the minimizer? The following Theorem proves that it is not the case.
It is the main result of this paper. The proof is absolutely not
obvious and will be given at the next section. Let us remark that
the following result also holds for the Stokes equation. The proof
in the Stokes case follows the same lines and is a little bit
simpler, see \cite{privat} for details.
\begin{theo}\label{MainTheoNSCylindre-en}
The cylinder is not the solution of the shape optimization problem
\begin{equation}\label{PbOptimS}
\left\{\begin{array}{l}
\min J(\Omega )\\
\Omega \in \mathcal{O}_{V},
\end{array}\right.
\end{equation}
where $J$ is defined in (\ref{crit}) with $\mathbf{u}$ the velocity,
solution of the Navier-Stokes problem (\ref{NSinit}), and
$\mathcal{O}_{V}$ is defined in (\ref{class1}).
\end{theo}
\section{Proof of the main theorem}\label{sec3}
In all this section, $\Om$ will now denote the cylinder
$\{x_1^2+x_2^2<R^2, 0<x_3<L\}$.
\subsection{Computation of the shape derivative}
Let us consider a regular vector field $\mathbf{ V}:\R^3\to\R^3$
with compact support in the strip $0<x_3<L$. For small $t$, we
define  $\Omega _t=(I+t\mathbf{ V})\Omega$, the image of $\Om$ by a
perturbation of identity and $f(t):=J(\Omega _t)$. We recall that
the shape derivative of $J$ at $\Om$ with respect to $\mathbf{ V}$
is $f'(0)$. We will denote it by $\d J(\Om;\mathbf{ V})$. To compute
it, we first need to compute the derivative of the state equation.
We use here the classical results of shape derivative as in
\cite{HP}, \cite{Mu-Si}, \cite{So-Zo}. The derivative of
$(\mathbf{u},p)$ is the solution of the following linear system:
\begin{equation}
\left\{\begin{array}{ll} \label{PDérivé}
\displaystyle -\mu \triangle \mathbf{u'} +\nabla \mathbf{u}\cdot \mathbf{u'}
+\nabla \mathbf{u'}\cdot \mathbf{u}+\nabla p'=0 & \mathbf{x}\in \Omega \\
\divgce\, \mathbf{u'}=0 & \mathbf{x}\in \Omega \\
\mathbf{u'}=\mathbf{0} & \mathbf{x}\in E \\
\displaystyle  \mathbf{u'}=-\frac{\partial \mathbf{u}}{\partial \mathbf{n}}
(\mathbf{V}\cdot \mathbf{n}) & \mathbf{x}\in \Gamma \\
\displaystyle -p'\mathbf{n}+2\mu \varepsilon (\mathbf{u'})\cdot
\mathbf{n}=0 & \mathbf{x}\in S.
\end{array}
\right.
\end{equation}

Now, we have (see \cite{HP}, \cite{So-Zo})
\begin{equation}\label{deriv1}
    \d J (\Omega , \mathbf{V})=4\mu \int_\Omega \varepsilon
(\mathbf{u}):\varepsilon (\mathbf{u'})\d x+2\mu \int_\Gamma
|\varepsilon (\mathbf{u})|^2(\mathbf{ V}\cdot \mathbf{n})\d s.
\end{equation}

\par It is more convenient to work with another expression of the shape
derivative. For that purpose, we need to introduce an adjoint state.
\begin{prop}\label{ExistenceRegulariteAdjoint}
\par Let us consider $(\mathbf{v},q)$, solution of the following adjoint
problem :
\begin{equation}
\left\{\begin{array}{ll}\label{PAdjoinNS}
\displaystyle -\mu \triangle \mathbf{v} +\nabla \mathbf{u}\cdot
\mathbf{v}-\nabla \mathbf{v}\cdot \mathbf{u}+\nabla q =-2\mu
\triangle \mathbf{u} & \mathbf{x}\in \Omega \\
{\rm div} \,\mathbf{v}=0 & \mathbf{x}\in \Omega \\
\mathbf{v}=\mathbf{0} & \mathbf{x}\in E\cup \Gamma \\
-q\mathbf{n}+2\mu \varepsilon (\mathbf{v})\cdot
\mathbf{n}+(\mathbf{u}\cdot \mathbf{n})\mathbf{v}-4\mu \varepsilon
(\mathbf{u})\cdot \mathbf{n}=0 & \mathbf{x}\in S.
\end{array}
\right.
\end{equation}
If the viscosity $\mu$ is large enough, then the problem
(\ref{PAdjoinNS}) has a unique solution $(\mathbf{v},q)$. Moreover,
this solution belongs to $C^1(\overline{\Omega})\times
C^0(\overline{\Omega}) $.
\end{prop}
\begin{proof}
The existence and uniqueness of the solution is a standard
application of Lax-Milgram's lemma. We introduce the Hilbert space
$$
V(\Omega ) :=\{\mathbf{u}\in H^1(\Omega ) : \divgce \mathbf{u}=0\}.
$$
the bilinear form $\alpha$ and the linear form $\ell$ defined by
\begin{eqnarray*}
\alpha (\mathbf{v},\mathbf{w}) & := & \int_\Omega
\left(2\mu \varepsilon (\mathbf{v}):\varepsilon (\mathbf{w})+\nabla
\mathbf{w}\cdot \mathbf{u}\cdot \mathbf{v}+\nabla \mathbf{u}\cdot
\mathbf{w}\cdot \mathbf{v}\right)\d x\\
\langle \ell , \mathbf{w}\rangle & := & 4\mu \int _\Omega
\varepsilon (\mathbf{u}): \varepsilon (\mathbf{w})\d x.
\end{eqnarray*}
To prove ellipticity of the bilinear form $\alpha$ we use Korn's
inequality:
$$
\Arrowvert \nabla \mathbf{v}\Arrowvert _{[L^2(\Omega )]^3}\leq
C_1(\Arrowvert \mathbf{v}\Arrowvert _{[L^2(\Omega )]^3}+ \Arrowvert
\varepsilon (\mathbf{v})\Arrowvert _{[L^2(\Omega )]^3}).
$$
and a Poincaré inequality:
\begin{equation}\label{inegPoinca}
\Arrowvert \mathbf{v}\Arrowvert _{[L^2(\Omega )]^3}\leq C_2
\int_\Omega |\varepsilon (\mathbf{v})|^2\d x.
\end{equation}
These two inequalities yield (we also use the explicit expression of
$\mathbf{u}$ given in (\ref{expli}) to estimate the integrals
containing $\mathbf{u}$):
$$
\alpha (\mathbf{v},\mathbf{v})\geq \left(\mu \frac{\min
(1,C_2)}{C_1+1}-|c|(R^2+2R)\right)\Arrowvert \mathbf{v}\Arrowvert
_{[H^1(\Omega )]^3}^2.
$$
and $\alpha$ is elliptic as soon as $\mu
>\frac{|c|(R^2+2R)(C_1+1)}{\min (1,C_2)}$. Now, existence and
uniqueness of the solution follow from a standard application of
Lax-Milgram's lemma together with De Rham's lemma to recover the
pressure.

It remains to prove the regularity of the solution. The $C^\infty$
regularity in $\Omega$ on the one-hand and on the smooth surfaces
$E$, $S$ and the interior of the lateral boundary $\Gamma$ on the
other hand is standard (cf. \cite{galdi}). The only point which is
not clear is the $C^1$ regularity on the circles ${E}\cap
\overline{\Gamma}$ and ${S}\cap \overline{\Gamma}$. To prove it, one
can use the cylindrical symmetry which is proved later (without any
regularity assumptions) in Theorem  \ref{symetrieEDP}. This symmetry
allows us to consider a two-dimensional problem in the rectangle
$(0,R)\times (0,L)$ into the variables $r=(x_1^2+x_2^2)^{1/2}$ and
$x_3$. For that problem, one need to prove regularity at the corners
$(R,0)$ and $(R,L)$. For that purpose, one extends the solution by
reflection around the line $r=R$, this leads to a partial
differential equation in the rectangle $(0,2R)\times (0,L)$ whose
solution coincides with our solution in the first half of the
rectangle. The $C^1$ regularity, up to the boundary, of the solution
of this elliptic p.d.e. is standard and the result follows.
\end{proof}

\medskip
Let us come back to the computation of the shape derivative. We
prove
\begin{prop}\label{deriv2}
With the previous notations, the shape derivative of the criterion
$J$ is given by
\begin{equation}
\label{DeriveeCritere} \d J (\Omega , \mathbf{V})=2\mu \int_\Gamma
\left(\varepsilon (\mathbf{u}):\varepsilon (\mathbf{v})-|\varepsilon
(\mathbf{u})|^2\right)(\mathbf{V}.\mathbf{n})\d s .
\end{equation}
\end{prop}
\begin{proof} Using Green's formula in (\ref{deriv1}), one gets
\begin{eqnarray*}
\d J (\Omega , \mathbf{V}) & = & 4\mu \int_\Omega \varepsilon
(\mathbf{u}):\varepsilon (\mathbf{u'})\d x+2\mu \int_\Gamma
|\varepsilon (\mathbf{u})|^2(\mathbf{ V}.\mathbf{n})\d s\\
 & = & -2\mu \int_\Omega ((\triangle \mathbf{u}+\nabla \divgce
 \mathbf{u})\cdot \mathbf{u}' )\d x+4\mu \int_{\partial \Omega }
 \varepsilon (\mathbf{u})\cdot \mathbf{n}\cdot \mathbf{u}'\d s \\
  & & +2\mu \int_{\partial \Omega }|\varepsilon (\mathbf{u})|^2
  (\mathbf{V}\cdot \mathbf{n})\d s
\end{eqnarray*}
Now, let us multiply the first equation of the adjoint problem
(\ref{PAdjoinNS}) by $\mathbf{u}'$ and integrate over $\Omega$, one
obtains
\begin{eqnarray*}
-\mu \int_\Omega \triangle \mathbf{v}\cdot \mathbf{u'}\d x+\int_\Omega
\nabla q\cdot \mathbf{u'} \d x +\int_\Omega (\nabla \mathbf{u})^T\cdot
\mathbf{v}\cdot \mathbf{u'}\d x & & \\
-\int_\Omega \nabla \mathbf{v}\cdot \mathbf{u}\cdot \mathbf{u'} \d x
= -2\mu \int_\Omega \triangle \mathbf{u}\cdot \mathbf{u'}\d x. & &
\end{eqnarray*}
Using one integration by parts and the boundary conditions satisfied
by $\mathbf{u}'$ and $\mathbf{v}$, we get
\begin{eqnarray*}
\int_\Omega \left(2\mu \varepsilon (\mathbf{u'}) \cdot \varepsilon
(\mathbf{v}) -\nabla \mathbf{v}\cdot \mathbf{u'} \cdot
\mathbf{u}+\nabla \mathbf{u'}\cdot \mathbf{u}\cdot
\mathbf{v}\right)\d x\\
-\int_S \sigma (\mathbf{v},q)\cdot \mathbf{n}\cdot \mathbf{u'}\d s
+\int_S \left((\mathbf{u}\cdot \mathbf{v})(\mathbf{u'}\cdot
\mathbf{n})-(\mathbf{u}\cdot \mathbf{n})(\mathbf{u'}\cdot
\mathbf{v})\right)\d s\\
-\int_\Gamma \sigma (\mathbf{v},q)\cdot \mathbf{n}) \cdot
\mathbf{u'} \d s= -2\mu \int_\Omega \triangle \mathbf{u}\cdot
\mathbf{u'}\d x.
\end{eqnarray*}
In the same way, if we multiply the first equation of the problem
(\ref{PDérivé}) by $\mathbf{v}$ and integrate over $\Omega$, we
obtain
\begin{eqnarray*}
-\mu \int_\Omega \triangle \mathbf{u'}\cdot \mathbf{v} \d
x+\int_\Omega \nabla p'\cdot \mathbf{v}\d x+\int_\Omega \nabla
\mathbf{u'}\cdot \mathbf{u}\cdot \mathbf{v}\d x +\int_\Omega \nabla
\mathbf{u}\cdot \mathbf{u'}\cdot \mathbf{v} \d x = 0
\end{eqnarray*}
and
\begin{eqnarray*}
\int_\Omega \left(2\mu \varepsilon (\mathbf{u'}) \cdot
\varepsilon (\mathbf{v}) +\nabla \mathbf{u'}\cdot \mathbf{u}\cdot
\mathbf{v}-\nabla \mathbf{v}\cdot \mathbf{u'}\cdot \mathbf{u}\right)\d x & & \\
+\int_S \left(-\sigma (\mathbf{u'},p')\cdot \mathbf{n}\cdot
\mathbf{v}+(\mathbf{u}\cdot \mathbf{v})(\mathbf{u'}\cdot
\mathbf{n})\right)\d s=0. & &
\end{eqnarray*}
Coming back to the shape derivative expression
\begin{eqnarray*}
\d J (\Omega , \mathbf{V}) & = & -2\mu \int_\Omega ((\triangle
\mathbf{u}+\nabla \divgce \mathbf{u})\cdot \mathbf{u}' )\d x+4\mu
\int_{\partial \Omega }\varepsilon (\mathbf{u})\cdot \mathbf{n}\cdot \mathbf{u}'\d s\\
 & & +2\mu \int_{\partial \Omega }|\varepsilon (\mathbf{u})|^2(\mathbf{V}\cdot \mathbf{n})\d s \\
 & = & A+4\mu \int_{\partial \Omega }\varepsilon (\mathbf{u})\cdot
 \mathbf{n}\cdot \mathbf{u}'\d s+2\mu \int_{\partial \Omega }|\varepsilon
 (\mathbf{u})|^2(\mathbf{V}\cdot \mathbf{n})\d s,
\end{eqnarray*}
where we set $\displaystyle A:=-2\mu \int_\Omega ((\triangle
\mathbf{u}+\nabla \divgce \mathbf{u})\cdot \mathbf{u}' )\d x$. Using
the previous identities, we get for $A$
\begin{eqnarray*}
A & = & \int_{\Gamma \cup S}\left(q\mathbf{n}-2\mu \varepsilon
(\mathbf{v})\cdot \mathbf{n}\right)\cdot \mathbf{u}'\d s-\int_S
(\mathbf{u}\cdot \mathbf{n})(\mathbf{v}\cdot \mathbf{u}')\d s.
\end{eqnarray*}
Therefore, according to (\ref{PAdjoinNS})
\begin{eqnarray*}
\d J (\Omega , \mathbf{V}) = \int_{\Gamma \cup S}\left(q\mathbf{n}-
2\mu \varepsilon (\mathbf{v})\cdot \mathbf{n}\right)\cdot
\mathbf{u}'\d s
-\int_S (\mathbf{u}\cdot \mathbf{n})(\mathbf{v}\cdot \mathbf{u}')\d s \\
 +4\mu \int_{S\cup \Gamma} \varepsilon (\mathbf{u})\cdot \mathbf{n}
 \cdot \mathbf{u}'\d s+2\mu \int_\Gamma |\varepsilon (\mathbf{u})|^2
 (\mathbf{ V}.\mathbf{n})\d s\\
 = \int_{\Gamma}\left(q\mathbf{n}-2\mu \varepsilon (\mathbf{v})\cdot
 \mathbf{n}+4\mu \varepsilon (\mathbf{u})\cdot \mathbf{n}\right)\cdot
 \mathbf{u}'\d s+2\mu \int_\Gamma |\varepsilon (\mathbf{u})|^2(\mathbf{ V}.\mathbf{n})\d s\\
 = -\int_{\Gamma}\left(\left(q\mathbf{n}-2\mu \varepsilon (\mathbf{v})
 \cdot \mathbf{n}+4\mu \varepsilon (\mathbf{u})\cdot \mathbf{n}\right)\cdot
 \frac{\partial \mathbf{u}}{\partial n}-2\mu |\varepsilon (\mathbf{u})|^2\right)
 (\mathbf{V}\cdot \mathbf{n})\d s
\end{eqnarray*}
To get the (more symmetric) expression given in
(\ref{DeriveeCritere}), one can use the following elementary
properties. Since $\mathbf{u}$ (and $\mathbf{v}$) is divergence-free
and vanishes on
  $\Gamma$, we have on this boundary:
\begin{itemize}
  \item  $\mathbf{n}\cdot
\frac{\partial \mathbf{u}}{\partial n}=0$.
  \item $\varepsilon
(\mathbf{u}) \cdot \mathbf{n}\cdot \frac{\partial
\mathbf{u}}{\partial \mathbf{n}}= |\varepsilon (\mathbf{u})|^2$.
  \item $(\varepsilon (\mathbf{v})\cdot \mathbf{n})\cdot \frac{\partial
\mathbf{u}}{\partial n}=\varepsilon (\mathbf{u}):\varepsilon
(\mathbf{v})$.
\end{itemize}
Proposition \ref{deriv2} follows. \end{proof}
\subsection{Analysis of the PDE (\ref{PAdjoinNS})}
We will prove the following symmetry result for the solution of the
adjoint system. It shows that the solution has the same symmetry as
the cylinder.
\begin{lemma}\label{symetrieEDP}\ \\
With the same assumptions on $\mu$ as in Proposition
\ref{ExistenceRegulariteAdjoint}, there exist $(w,w_3)\in
[H^1((0,R)\times (0,L))]^2$ and $\tilde{q}\in L^2((0,R)\times
(0,L))$ such that, for any $(x_1,x_2,x_3)\in \Omega$
\begin{description}
\item[(i)] $v_i(x_1,x_2,x_3)=x_i w(r,x_3)$, for $i\in \{1,2\}$ ;
\item[(ii)] $v_3(x_1,x_2,x_3)= w_3(r,x_3)$ ;
\item[(iii)] $q(x_1,x_2,x_3)= \tilde{q}(r,x_3)$.
\end{description}
where $r=(x_1^2+x_2^2)^{1/2}$.
\end{lemma}
\begin{proof}
Let us introduce the differential operator $\mathfrak{L}_\theta$
defined by
$$
\mathfrak{L}_\theta =x_1\frac{\partial }{\partial
x_2}-x_2\frac{\partial }{\partial x_1}.
$$
$\mathfrak{L}_\theta $ corresponds actually to the differentiation
with respect to the polar angle $\theta$. Let us set
\begin{equation}
\widehat{v_i}=\mathfrak{L}_\theta (v_i), \ \forall i\in
\{1,2,3\}\textrm{ and }\widehat{q}=\mathfrak{L}_\theta (q).
\end{equation}
By applying the operator $\mathfrak{L}_\theta $ to the equation
(\ref{PAdjoinNS}) we get the following system (where we have used
the explicit expression of the solution $\mathbf{u}$ given in
(\ref{expli}))
\begin{equation}\label{HatAdjoinNS}
\left\{\begin{array}{ll}
\displaystyle -\mu \triangle \widehat{v_1} +2cx_1\widehat{v_3}-2cx_2
v_3-c(x_1^2+x_2^2-R^2)\frac{\partial \widehat{v_1}}{\partial x_3}+
\frac{\partial \widehat{q}}{\partial x_1}-\frac{\partial q}{\partial x_2}
=0 & \mathbf{x}\in \Omega \\
\displaystyle -\mu \triangle \widehat{v_2} +2cx_2\widehat{v_3}+2cx_1v_3
-c(x_1^2+x_2^2-R^2)\frac{\partial \widehat{v_2}}{\partial x_3}+
\frac{\partial \widehat{q}}{\partial x_2}+\frac{\partial q}{\partial x_1}
=0 & \mathbf{x}\in \Omega \\
\displaystyle -\mu \triangle \widehat{v_3} -c(x_1^2+x_2^2-R^2)
\frac{\partial \widehat{v_3}}{\partial x_3}+
\frac{\partial \widehat{q}}{\partial x_3}=0 & \mathbf{x}\in \Omega \\
\displaystyle \frac{\partial \widehat{v_1}}{\partial x_1}+
\frac{\partial \widehat{v_2}}{\partial x_2}+
\frac{\partial \widehat{v_3}}{\partial x_3}-\frac{\partial v_1}{\partial x_2}+
\frac{\partial v_2}{\partial x_1}=0 & \mathbf{x}\in \Omega \\
\widehat{v_1}=\widehat{v_2}=\widehat{v_3}=0 & \mathbf{x}\in E\cup \Gamma \\
\displaystyle \mu \left(\frac{\partial \widehat{v_1}}{\partial x_3}+
\frac{\partial \widehat{v_3}}{\partial x_1}\right)-\mu
\frac{\partial v_3}{\partial x_2}+c(x_1^2+x_2^2-R^2)\widehat{v_1}=-4\mu
c x_2 & \mathbf{x}\in S, \\
\displaystyle \mu \left(\frac{\partial \widehat{v_2}}{\partial x_3}+
\frac{\partial \widehat{v_3}}{\partial x_2}\right)+\mu
\frac{\partial v_3}{\partial x_1}+c(x_1^2+x_2^2-R^2)\widehat{v_2}=4\mu
c x_1 & \mathbf{x}\in S, \\
\displaystyle 2\mu \frac{\partial \widehat{v_3}}{\partial
x_3}+c(x_1^2+x_2^2-R^2)\widehat{v_3}=\widehat{q} & \mathbf{x}\in S,
\end{array}
\right.
\end{equation}
Let us now introduce the following new functions
\begin{itemize}
\item $z_1=\widehat{v_1}+v_2$ ;
\item $z_2=\widehat{v_2}-v_1$ ;
\item $z_3=\widehat{v_3}$.
\end{itemize}
According to system (\ref{PAdjoinNS}), the system
(\ref{HatAdjoinNS}) rewrites in term of $z_1$, $z_2$, $z_3$
\begin{equation}
\left\{\begin{array}{ll}
\displaystyle -\mu \triangle z_1 +2cx_1z_3-c(x_1^2+x_2^2-R^2)
\frac{\partial z_1}{\partial x_3}+\frac{\partial \widehat{q}}{\partial x_1}=0 &
\mathbf{x}\in \Omega \\
\displaystyle -\mu \triangle z_2 +2cx_2z_3-c(x_1^2+x_2^2-R^2)
\frac{\partial z_2}{\partial x_3}+\frac{\partial \widehat{q}}{\partial x_2}=0 & \mathbf{x}\in \Omega \\
\displaystyle -\mu \triangle z_3 -c(x_1^2+x_2^2-R^2)\frac{\partial z_3}{\partial x_3}+
\frac{\partial \widehat{q}}{\partial x_3}=0 & \mathbf{x}\in \Omega \\
\displaystyle \frac{\partial z_1}{\partial x_1}+\frac{\partial z_2}{\partial x_2}+
\frac{\partial z_3}{\partial x_3}=0 & \mathbf{x}\in \Omega \\
z_1=z_2=z_3=0 & \mathbf{x}\in E\cup \Gamma \\
\displaystyle \mu \left(\frac{\partial z_1}{\partial x_3}+
\frac{\partial z_3}{\partial x_1}\right)+z_1c(x_1^2+x_2^2-R^2)=0 & \mathbf{x}\in S, \\
\displaystyle \mu \left(\frac{\partial z_2}{\partial x_3}+
\frac{\partial z_3}{\partial x_2}\right)+z_2c(x_1^2+x_2^2-R^2)=0 & \mathbf{x}\in S, \\
\displaystyle 2\mu \frac{\partial z_3}{\partial
x_3}+c(x_1^2+x_2^2-R^2)z_3=\widehat{q} & \mathbf{x}\in S,
\end{array}
\right.
\end{equation}
This adjoint problem has a unique solution if $\mu$ is large enough
(see proposition \ref{ExistenceRegulariteAdjoint}), therefore
$$
z_1=z_2=\widehat{v_3}=\widehat{q}\equiv 0.
$$
The fact that $\widehat{v_3}=\mathfrak{L}_\theta (v_3)$ and
$\widehat{q}=\mathfrak{L}_\theta (q)$ vanish proves points ii and
iii of the Lemma. Now let us precise the properties of functions
$v_1,v_2$. It has been proved that $\mathfrak{L}_\theta (v_1)=-v_2$
and $\mathfrak{L}_\theta (v_2)=v_1$. Therefore, applying once more
the operator $\mathfrak{L}_\theta $ yields $\mathfrak{L}_\theta
\circ \mathfrak{L}_\theta (v_1)+v_1=0$. This implies that there
exist two functions $\alpha $ and $\beta$ in the space
$H^1((0,R)\times (0,L))$, such that
$$
v_1=x_1 \alpha (r,x_3)+x_2\beta (r,x_3).
$$
Moreover, since $\mathfrak{L}_\theta (v_1)=-v_2$, we get
$$
v_2=-x_1\beta (r,x_3)+x_2\alpha (r,x_3).
$$
To finish the proof, it remains to check that the function $\beta$
is identically zero. For that purpose, let us write down the partial
differential equation satisfied by $\beta$. From the two first
equations of system (\ref{PAdjoinNS}) and the boundary condition, we
can prove that $\beta$ satisfies the following system
\begin{equation}\label{equationBeta}
\left\{\begin{array}{ll}
\displaystyle -\mu \left(\frac{\partial ^2 \beta }{\partial r^2}+
\frac{3}{r}\frac{\partial \beta }{\partial r}+
\frac{\partial ^2\beta }{\partial x_3^2}\right)-c(r^2-R^2)
\frac{\partial \beta }{\partial x_3}=0 & (r,x_3)\in (0,R)\times (0,L)\\
\displaystyle \beta (r,0)=\beta (R,x_3)=
\frac{\partial \beta }{\partial r}(0,x_3)=0 & (r,x_3)\in (0,R)\times (0,L)\\
\displaystyle \mu \frac{\partial \beta }{\partial n}+c(r^2-R^2)\beta =0 &
(r,x_3)\in (0,R)\times \{L\}\\
\end{array}\right.
\end{equation}
It remains to prove that the zero function is the unique solution of
the previous system. Multiplying the equation by $\beta$ and
integrating on the rectangle in polar coordinates gives, using the
boundary conditions
$$
    0=\mu \int_\Omega \left(\left(\frac{\partial \beta}{\partial
r}\right)^2+\left(\frac{\partial \beta}{\partial
x_3}\right)^2\right)r\d r \d x_3 + $$ $$    +\mu \int_0^L\beta
^2(0,x_3) \d x_3 +\frac{c}{2}\int_0^R (r^2-R^2)\beta ^2(r,L)r\d r.
$$
Since $c<0$ and $r<R$, we get $\frac{\partial \beta}{\partial
r}\equiv 0$ in $(0,R)\times (0,L)$ and $\beta ^2(0,x_3)=0$ for any
$x_3\in (0,L)$. Then $\beta \equiv 0$ which gives the desired
result.
\end{proof}
\subsection{The optimality condition}
We argue by contradiction. Let us assume that the cylinder $\Om$ is
optimal for the criterion $J$. We first write down the first order
optimality condition. From the explicit expression (\ref{expli}) of
$\mathbf{u}$, we have
$$
\varepsilon(\mathbf{u})=\left(\begin{array}{ccc}
0 & 0 & cx_1 \\
0 & 0 & cx_2 \\
cx_1 & cx_2 & 0
\end{array}\right).
$$
Therefore
$$
|\varepsilon(\mathbf{u})|^2=2c^2(x_1^2+x_2^2),
$$
and $|\varepsilon(\mathbf{u})|^2=2c^2R^2$ is constant on $\Gamma$.
\par Now the first order optimality condition ensures the existence
of a Lagrange multiplier $\lambda \in \R$, such that $\d J(\Omega ,
\mathbf{V})=\lambda\, \d \textrm{Vol }(\Omega ,\mathbf{V})$ for any
vector field $\mathbf{V}$. Due to the expression of the shape
derivatives of $J$ and the volume, it writes
$$
2\mu\int_\Gamma  \left(\varepsilon (\mathbf{u}):\varepsilon
(\mathbf{v})-|\varepsilon
(\mathbf{u})|^2\right)(\mathbf{V}.\mathbf{n}) \d s=\lambda
\int_\Gamma (\mathbf{V}\cdot \mathbf{n})\d s.
$$
This implies that $ \varepsilon (\mathbf{u}):\varepsilon
(\mathbf{v})$ is constant on $\Gamma$. Now, from the expression of
$\varepsilon (\mathbf{u})$ on $\Gamma$, we deduce
\begin{eqnarray*}
\varepsilon (\mathbf{u}):\varepsilon (\mathbf{v})_{\mid _\Gamma} &
= & \frac{c}{2}\left(x_1\frac{\partial v_3}{\partial x_1}+x_2
\frac{\partial v_3}{\partial x_2}+x_1\frac{\partial v_1}{\partial x_3}+
x_2\frac{\partial v_2}{\partial x_3}\right) \\
 & = & \frac{c}{2}\left(x_1\frac{\partial v_3}{\partial x_1}+x_2
 \frac{\partial v_3}{\partial x_2}\right)=\frac{cR}{2}
 \frac{\partial v_3}{\partial n}_{\mid _\Gamma},
\end{eqnarray*}
because ${v_1}_{\mid _\Gamma}={v_2}_{\mid _\Gamma}=0$. Therefore the
optimality condition writes
\begin{equation}\label{CionOptim1}
\exists \xi \in \R : \frac{\partial v_3}{\partial n}=\xi
\textnormal{ on }\Gamma .
\end{equation}

\medskip
Now, we give another useful Lemma
\begin{lemma}\label{qGamma}
If the cylinder $\Om$ is optimal and using the notations of Lemma
\ref{symetrieEDP}, we have
$$
\frac{\partial q}{\partial n}_{\mid _\Gamma}=\frac{\partial
\tilde{q}}{\partial r}_{\mid _{\{r=R\}}}=0.
$$
\end{lemma}
\begin{proof}
Let us write the adjoint problem (\ref{PAdjoinNS}) in term of the
functions $w$, $w_3$ et $\tilde{q}$. We get
\begin{equation}\label{NewPAdjoinNS}
\left\{\begin{array}{ll} \displaystyle -\mu \left(\frac{\partial
^2w}{\partial r^2}+\frac{1}{r} \frac{\partial w}{\partial
r}+\frac{\partial ^2w}{\partial x_3 ^2}\right)+
\frac{1}{r}\frac{\partial \tilde{q}}{\partial r}+2cw_3-c(r^2-R^2)
\frac{\partial w}{\partial x_3} =0  &\mbox{in $\Om$}\\
\displaystyle -\mu \left(\frac{\partial ^2w_3}{\partial r^2}+
\frac{1}{r}\frac{\partial w_3}{\partial r}+
\frac{\partial ^2w_3}{\partial x_3 ^2}\right)+\frac{1}{r}
\frac{\partial \tilde{q}}{\partial x_3}-c(r^2-R^2)
\frac{\partial w_3}{\partial x_3} =-8\mu c &\mbox{in $\Om$} \\
\displaystyle 2w+r\frac{\partial w}{\partial r}+
\frac{\partial w_3}{\partial x_3} =0  &\mbox{in $\Om$}\\
w(r,0)=w_3(r,0)=w(R,x_3)=w_3(R,x_3)=0  &\\
\displaystyle \mu \left(\frac{\partial w}{\partial x_3}+
\frac{1}{r}\frac{\partial w_3}{\partial r}\right)+c(r^2-R^2)w =
4\mu c  &\mbox{on $S$}\\
\displaystyle 2\mu \frac{\partial w_3}{\partial x_3}+c(r^2-R^2)w_3
=\tilde{q} &\mbox{on $S$}.
\end{array}
\right.
\end{equation}
Since $w_{\mid _{\{r=R\}}}={w_3}_{\mid _{\{r=R\}}}=0$, we have
$\frac{\partial w}{\partial x_3}_{\mid _{\{r=R\}}}=\frac{\partial
w_3}{\partial x_3}_{\mid _{\{r=R\}}}=0$ and $\frac{\partial
^2w}{\partial x_3^2}_{\mid _{\{r=R\}}}=0$. In particular, from the
divergence-free condition, we obtain $\frac{\partial w}{\partial
r}_{\mid _{\{r=R\}}}=0$.
\par Now, let us differentiate the divergence-free condition with respect to
$r$, we get
$$
\forall (r,x_3)\in (0,R)\times (0,L), \ 3\frac{\partial w}{\partial
r}+r\frac{\partial ^2w}{\partial r^2}+\frac{\partial^2w_3}{\partial
r\partial x_3}=0.
$$
Now, $ \frac{\partial w_3}{\partial r}_{\mid _{\{r=R\}}}=\xi$ (it is
the optimality condition (\ref{CionOptim1})) ; therefore, we have $
\frac{\partial^2w_3}{\partial x_3\partial r}_{\mid _{\{r=R\}}}=0$.
Combining this last result with $\frac{\partial w}{\partial r}_{\mid
_{\{r=R\}}}=0$, it comes
$$
\frac{\partial ^2w}{\partial r^2}_{\mid _{\{r=R\}}}=0.
$$
We let $r$ going to $R$ in the first equation of problem
(\ref{NewPAdjoinNS}) and we use the previous identities to get
$$
\frac{\partial \tilde{q}}{\partial r}_{\mid _{\{r=R\}}}=0.
$$
\end{proof}
\subsection{An auxiliary function}
Using notation of Lemma \ref{symetrieEDP}, we introduce now two new
functions
\begin{itemize}
\item $w_0 : \begin{array}[t]{rcl}
[0,R]\times [0, L] & \longrightarrow & \R \\
(r,x_3) & \longmapsto & \displaystyle \int_0^{x_3}w(r,z)\d z
\end{array}$.
\item $\psi : \begin{array}[t]{rcl}
[0,R]\times [0, L] & \longrightarrow & \R \\
x_3 & \longmapsto & \displaystyle \int_0^{R}\! \! \!
\int_0^{2\pi}\left(\tilde{q}(r,x_3)-2cr^2 w_0(r,x_3)\right)\d \theta
r\d r
\end{array}$.
\end{itemize}
We will also denote by  $T_z$ the horizontal section of the cylinder
$ \left\{\mathbf{x}\in \Omega : x_3=z\right\}$. The following lemma
is the key point of the proof.
\begin{lemma}\label{psiAffine}
The function $\psi$ is affine.
\end{lemma}
\begin{proof}
The couple $(\mathbf{v},q)$ satisfies the following p.d.e.
$$
-\mu \triangle \mathbf{v}+\nabla q +\nabla \mathbf{u}\cdot
\mathbf{v}-\nabla \mathbf{v}\cdot \mathbf{u}=-2\mu \triangle
\mathbf{u}.
$$
Let us compute the divergence of both sides of the previous
equality. Using the expression of $\mathbf{u}$ in the cylinder
$\Omega$, we obtain that $(\mathbf{v},q)$ verifies
\begin{equation}
\triangle q +4cv_3+2c \left(x_1\frac{\partial v_3}{\partial
x_1}+x_2\frac{\partial v_3}{\partial x_2}\right)-2c
\left(x_1\frac{\partial v_1}{\partial x_3}+x_2\frac{\partial
v_2}{\partial x_3}\right)=0.
\end{equation}
Let us integrate this equation on a slide
$$\omega:=\{(x_1,x_2,x_3)\in \Omega; z_- \leq x_3 \leq z_+\}$$ (we
will denote by  $e$ the inlet of $\omega$ and $s$ its outlet). We
get
$$
\int_\omega \triangle q +4c v_3\d x+2c \int_\omega
\left(x_1\frac{\partial v_3}{\partial x_1}+x_2\frac{\partial
v_3}{\partial x_2}\right)-2c  \left(x_1\frac{\partial v_1}{\partial
x_3}+x_2\frac{\partial v_2}{\partial x_3}\right)\d x =0.
$$
Now, from Green's formula, we have
$$\displaystyle \int_\omega
x_1\frac{\partial v_3}{\partial x_1}\d x =\int _{\partial \omega
}x_1v_3 n_1\d s-\int_\omega v_3 \d x= \int _{\partial \omega \cap
\Gamma }x_1v_3 n_1\d s-\int_\omega v_3 \d x=-\int_\omega v_3 \d x$$
$$ \mbox{in the same way }\; \displaystyle \int_\omega x_2
\frac{\partial v_3}{\partial x_2}\d x=-\int_\omega v_3 \d x\,.$$
Therefore
$$
4c\int_\omega v_3\d x+2c \int_\omega \left(x_1\frac{\partial
v_3}{\partial x_1}+x_2\frac{\partial v_3}{\partial x_2}\right)\d
x=0,
$$
so
\begin{equation}\label{psiPrime}
\int_\omega \triangle q \d x=2c \int_\omega \left(x_1\frac{\partial
v_1}{\partial x_3}+x_2\frac{\partial v_2}{\partial x_3}\right)\d x .
\end{equation}
Let us consider the left-hand side of (\ref{psiPrime}). From Lemma
\ref{qGamma} it comes
\begin{equation}\label{triangleQBord}
\int_\omega \triangle q \d x=\int_{s\cup e}\frac{\partial
q}{\partial n}\d s.
\end{equation}
Now, let us consider the right-hand side of (\ref{psiPrime}).
Integrating by parts yields
\begin{itemize}
\item $\displaystyle \int_\omega x_1\frac{\partial v_1}{\partial x_3}\d x
=\int_{\partial \omega}x_1v_1n_3 \d s=\int_{e\cup s}x_1v_1n_3\d s$.
\item $\displaystyle \int_\omega x_2\frac{\partial v_2}{\partial x_3}\d x
=\int_{\partial \omega}x_2v_2n_3 \d s=\int_{e\cup s}x_2v_2n_3\d s$.
\end{itemize}
Combining this result with (\ref{triangleQBord}) gives
\begin{equation}
\int_s \left(\frac{\partial q}{\partial x_3}-2c
(x_1v_1+x_2v_2)\right)\d s=\int_e \left(\frac{\partial q}{\partial
x_3}-2c (x_1v_1+x_2v_2)\right)\d s,
\end{equation}
what can also be rewritten for any $(z_-,z_+)\in (0,L)^2$ :
\begin{equation}\label{PsiPrime2}
\int_0^R\left(\frac{\partial \tilde{q}}{\partial x_3}(r,z_-)-2cr^2
w(r,z_-)\right)r\d r=\int_0^R\left(\frac{\partial
\tilde{q}}{\partial x_3}(r,z_+)-2cr^2 w(r,z_+)\right)r\d r.
\end{equation}
Now, since $\displaystyle \psi (z)=2\pi
\int_0^R\left(\tilde{q}(r,z)-2cr^2 w_0(r,z)\right)r\d r$, we have by
differentiating, for all $z$ in $[0,L],$
$$
  \psi ' (z)=2\pi \int_0^R\left(\frac{\partial \tilde{q}}{\partial
x_3}-2cr^2 \frac{\partial w_0}{\partial x_3}\right)r\d r=2\pi
\int_0^R\left(\frac{\partial \tilde{q}}{\partial x_3}-2cr^2
w\right)r\d r.
$$
Now, identity (\ref{PsiPrime2}) proves that $\psi '$ is a constant
function which gives the desired result.
\end{proof}

\medskip
We are now in position to precise the value of the constant $\xi$
appearing in the first order optimality condition
(\ref{CionOptim1}). For that purpose, we use the symmetry result
given in Lemma \ref{symetrieEDP} together with equation
(\ref{NewPAdjoinNS}). In this equation, let us integrate between
$x_3=0$ and $x_3=z \in (0,L)$. Since $w_3(r,0)=0$, we get for any
$(r,z)\in [0,R]\times [0,L]$ :
$$
2w_0(r,z)+r \frac{\partial w_0}{\partial r}(r,z)+w_3(r,z)=0.
$$
Let us differentiate this last relation with respect to $r$. This
yields
\begin{equation}\label{divgcew0}
3\frac{\partial w_0}{\partial r}+\frac{\partial ^2w_0}{\partial
r^2}+\frac{\partial w_3}{\partial r}=0.
\end{equation}
Now, in (\ref{NewPAdjoinNS}), we differentiate the divergence
equation with respect to $r$, and we make $r\to R$. We obtain
$$
\frac{\partial w}{\partial r}_{\mid _\Gamma}=\frac{\partial
^2w}{\partial r^2}_{\mid _\Gamma}=0.
$$
Letting  $r$ going to $R$ in (\ref{divgcew0}) and interverting limit
and integral gives, using the previous equality
$$
\frac{\partial v_3}{\partial n}_{\mid _\Gamma}=0.
$$
So we conclude that $\xi =0$ and the optimality condition rewrites
\begin{equation}\label{cionOptimNS}
\frac{\partial v_3}{\partial n}_{\mid _\Gamma}=0.
\end{equation}
\subsection{End of the proof}
Let us use the function $\psi$ defined above. We can rewrite it as
$$
\psi (z)=\int_{T_z} \left(\tilde{q}-2cr^2 w_0\right)\d \theta r\d r
= 2\pi \int_0^R\left(\tilde{q}(r,z)-2cr^2 w_0(r,z)\right)r\d r,
$$
where $T_z$ denotes the horizontal section of the cylinder of cote
$z$. We proved in Lemma \ref{psiAffine} that $\psi$ is affine,
therefore its derivative $\psi'$ is constant, say $\psi'(z)=a$. The
contradiction will come from the computation of this constant on the
inlet $E$ and the outlet $S$. We will see that we obtain two
different values. Let us denote by $\triangle _2$ the
two-dimensional Laplacian (with respect to the variables $x_1$ and
$x_2$).

\noindent\textit{Computation of the constant on the outlet $S$ of
the cylinder.} First of all, let us remark that if we differentiate
with respect to $x_1$ the boundary condition on $S$ satisfied by the
function $v_1$, we get
\begin{equation}\label{v1SNS}
\mu \frac{\partial ^2v_1}{\partial x_1\partial x_3}+\mu
\frac{\partial ^2v_3}{\partial
x_1^2}+2cx_1v_1+c(x_1^2+x_2^2-R^2)\frac{\partial v_1}{\partial
x_1}=4\mu c, \ \textrm{on }S.
\end{equation}
In the same way, if we differentiate with respect to $x_2$ the
boundary condition on $S$ satisfied by the function $v_2$, we get
\begin{equation}\label{v2SNS}
\mu \frac{\partial ^2v_2}{\partial x_2\partial x_3}+\mu
\frac{\partial ^2v_3}{\partial
x_2^2}+2cx_2v_2+c(x_1^2+x_2^2-R^2)\frac{\partial v_2}{\partial
x_2}=4\mu c, \ \textrm{on }S.
\end{equation}
Summing the two relations (\ref{v1SNS}) and (\ref{v2SNS}) and using
the divergence-free condition yields
$$
-\mu \frac{\partial ^2 v_3}{\partial x_3^2}+\mu \triangle
_2v_3+2c(x_1v_1+x_2v_2)-c(x_1^2+x_2^2-R^2)\frac{\partial
v_3}{\partial x_3}=8\mu c \ \textrm{on }S.
$$
Now, according to (\ref{PAdjoinNS}), $v_3$ satisfies
\begin{equation}\label{2DS}
    \mu \triangle _2v_3=8\mu c-\mu \frac{\partial ^2v_3}{\partial
x_3^2}-c(x_1^2+x_2^2-R^2)\frac{\partial v_3}{\partial
x_3}+\frac{\partial q}{\partial x_3}.
\end{equation}
Combining together the two previous equations, it comes
\begin{equation}\label{cionSupplSNS}
-2\mu \frac{\partial ^2v_3}{\partial
x_3^2}-2c(x_1^2+x_2^2-R^2)\frac{\partial v_3}{\partial
x_3}+\frac{\partial q}{\partial x_3}+2c(x_1v_1+x_2v_2)=0 \
\textrm{on }S.
\end{equation}
Now, we integrate on $S$ the equation (\ref{2DS}), we have
$$
\int_S \left(-\mu \triangle_2 v_3-\mu \frac{\partial ^2v_3}{\partial
x_3^2}-\frac{\partial v_3}{\partial
x_3}(x_1^2+x_2^2-R^2)c+\frac{\partial q}{\partial x_3}\right)\d s
=-8\mu c\int_S\d s.
$$
In the Proposition \ref{ExistenceRegulariteAdjoint}, we have seen
that $v_3$ is $C^1$ up to the boundary. Taking into account the
boundary condition on $S$, we have
$$
\displaystyle \int_S \triangle_2 v_3\d s=\int_{S\cap \Gamma}
\frac{\partial v_3}{\partial n}\d \sigma =0\,.
$$
So, the integration gives
$$
-\mu \int_S\frac{\partial ^2v_3}{\partial x_3^2}\d
s-c\int_S(x_1^2+x_2^2-R^2)\frac{\partial v_3}{\partial x_3}\d
s+\int_S\frac{\partial q}{\partial x_3}\d s =-8\mu c\pi R^2.
$$
Using (\ref{cionSupplSNS}), we can deduce that
$$
\frac{1}{2}\int_S\frac{\partial q}{\partial x_3}\d
s-c\int_S(x_1v_1+x_2v_2)\d s=-8\mu c\pi R^2.
$$
According to Lemma \ref{symetrieEDP}, one can write
$$
x_1v_1+x_2v_2=(x_1^2+x_2^2)w\left(\left(x_1^2+x_2^2\right)^{1/2},x_3\right).
$$
Therefore
\begin{equation}\label{psi'(L)}
a=\psi '(L)=-16\mu c \pi R^2
\end{equation}

\noindent\textit{Computation of the constant on the inlet $E$ of the
cylinder.} Let us first remark that $\displaystyle \frac{\partial
v_3}{\partial x_3}_{\mid _{E}}=0$ (just use the divergence-free
condition extended to $E$ and the fact that ${v_1}_{\mid
_E}={v_2}_{\mid _E}=0$). Let us now integrate the p.d.e.
(\ref{PAdjoinNS}) satisfied by $v_3$. We have, using $\frac{\partial
v_3}{\partial x_3}_{\mid _{E}}=0$,
$$
-\mu \int_E \triangle v_3 \d s+\int_E\frac{\partial q}{\partial
x_3}\d s=-8\mu c\int_E \d s.
$$
Taking into account the condition (\ref{cionOptimNS}) we get
\begin{eqnarray*}
-\mu \int_E \triangle v_3 \d s & = & -\mu \int_E \triangle _2v_3\d s
-\mu \int_E\frac{\partial ^2v_3}{\partial x_3^2}\d s \\
 & = & -\mu \int_{E\cap \Gamma}\frac{\partial v_3}{\partial n}\d \sigma +
 \mu \int_E \left(\frac{\partial ^2v_1}{\partial x_3\partial x_1}+
 \frac{\partial ^2v_2}{\partial x_3\partial x_2}\right)\d s\\
 & = & \mu \int_{E\cap \Gamma}\left(\frac{\partial v_1}{\partial x_3}n_1
 +\frac{\partial v_2}{\partial x_3}n_2\right)\d \sigma =0.
\end{eqnarray*}
Then, it follows
\begin{equation}\label{qOnE}
\int_E\frac{\partial q}{\partial x_3}\d s=-8\mu c \pi R^2.
\end{equation}
At last, since ${v_1}_{\mid _E}={v_2}_{\mid _E}=0$, we have
$$
\psi '(0)=2\pi \int_0^R\left(\frac{\partial \tilde{q}}{\partial
z}(r,0)-2cr^2 w(r,0)\right)r\d r=\int_E\frac{\partial q}{\partial
x_3}\d s.
$$
According to (\ref{qOnE}) we have
\begin{equation}\label{psi'(0)}
a=\psi '(0)=-8\mu c \pi R^2.
\end{equation}
which is clearly a contradiction with (\ref{psi'(L)}) since $c<0$.
This finishes the proof of Theorem \ref{MainTheoNSCylindre-en}.
\section{Some numerical results}\label{sec4}
In this section are presented some numerical computations. It gives
a confirmation that the cylinder is not an optimal shape for the
problem of minimizing the dissipated energy. In particular, we are
able to exhibit better shapes for this criterion. All these
computations have been realized with the software Comsol.
\par For any bounded, simply connected domain $\Om$ in $\mathbb{R}^2$ or
$\mathbb{R}^3$ and any real numbers $\mu,b$ ($b$ will be fixed in
all the algorithm), let us define the augmented Lagrangian of our
problem (\ref{PbOptimS}) by
$$
\mathcal{L}(\Omega,\mu)=J(\Omega)+\mu(|\Omega|-V)+\frac{b}{2}\,(|\Omega|-V)^2.
$$
\par Since Theorem \ref{MainTheoNSCylindre-en} ensures that the cylinder
is not optimal for the criterion $J$, the question of finding a
better shape in the class of admissible domains
$\mathcal{O}^{\varepsilon}_V$ is natural. The numerical difficulties
in such a work, are the non linear character of the state equation
and the need to take into account the volume constraint.
\par For that reason, we decompose the work in two steps. First, is
considered a gradient type algorithm in two dimensions which allows
us to reduce the criterion $J$. Then, we work in a three dimensional
class of domains with constant volume $V$ and cylindrical symmetry.
In this class, we are able to find a shape (probably not optimal)
which is better than the cylinder, see section \ref{3Dc}.
\subsection{A numerical algorithm in 2D}
We denote by $\Omega_0$ the cylinder with inlet $E$, outlet $S$, and
measure $V$. $\Omega_0$ is our initial guess for the gradient type
algorithm we consider. We deform $\Omega_0$ by using the following
method:
\begin{enumerate}
\item We fix $\mu_0\in\mathbb{R}$, $\tau>0$ and $\varepsilon>0$.
\item Iteration $m$. At the previous iteration, $\mu_m$ and $\Omega_m$
have been computed. We define
$\Omega_{m+1}:=(I+\varepsilon_m\mathbf{d_m})(\Omega_m)$, where $I$
denotes the identity operator, $\varepsilon_m$ is a real number
(step of the gradient method) which is determined through a
classical 1D optimization method and $\mathbf{d_m}$ is a vector
field of $\mathbb{R}^2$, solution of the p.d.e.
$$
\left\{\begin{array}{ll}
-\triangle \mathbf{d_m}+\mathbf{d_m}=0 & \mathbf{x}\in \Omega_m\\
\mathbf{d_m}=0 & \mathbf{x}\in E\cup S\\
\frac{\partial \mathbf{d_m}}{\partial n}=-\frac{\partial
\mathcal{L}}{\partial n}\mathbf{n} & \mathbf{x}\in \Gamma_m,
\end{array}\right.
$$
where $\Gamma_m$ denotes the lateral boundary of $\Omega_m$, i.e.
$\Gamma_m:=\partial\Omega_m\backslash(E\cup S)$. The solution of
this p.d.e. gives a descent direction for the criterion $J$ (see for
instance \cite{allaire}, \cite{dogan}).
\par Then, the Lagrange multiplier $\mu_m$ is actualized by setting
$$
\mu_{m+1}:=\mu_m+\tau(|\Omega_{m+1}|-V).
$$
\item We stop the algorithm when $(\mu_m)_{m\geq 0}$ has converged
and the derivative of the Lagrangian is small enough.
\end{enumerate}
The Figure \ref{gradientShape} shows the geometry we obtain. The
criterion has decreased about 1.1 $\%$ from the initial
configuration (a rectangle here).
\begin{figure}[!h]
\begin{center}
\includegraphics[width=7.5cm,height=6cm]{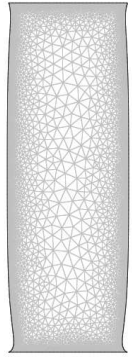}
\caption{Final 2-D shape obtained by the gradient
algorithm}\label{gradientShape}
\end{center}
\end{figure}
\subsection{Some 3D computations}\label{3Dc}
In this section, we create a family of 2D shapes, constructed with
cubic spline curves which look like the presumed optimum obtained in
figure \ref{gradientShape}. Then, we obtain a family of 3D domains
of volume $V$, by revolving the previous 2D shapes around the
$(Ox_3)$ axis. We introduce a small parameter $e$ in the control
points of the cubic splines and we evaluate for each value of $e$
the criterion $J$. The value $e=0$ corresponds to the cylinder. Let
us respectively denote by $J(e)$ and $J(\Omega_0)$ the values of the
criterion $J$ evaluated at the domain corresponding to value $e$ of
the parameter and at the cylinder. Figure \ref{ShapeVol} is the plot
of function $e\mapsto 100.\frac{J(e)-J(\Omega_0)}{J(\Omega_0)}$
above, and Figure \ref{ShapeVol2} represents a better shape than the
cylinder for the criterion $J$ which is obtained with a value of the
parameter $e\simeq 0.001$. It shows that this simple method provides
a 3D (axially symmetric) shape which is slightly better than the
cylinder.
\begin{figure}[!h]
\begin{center}
\includegraphics[width=8.5cm,height=7cm]{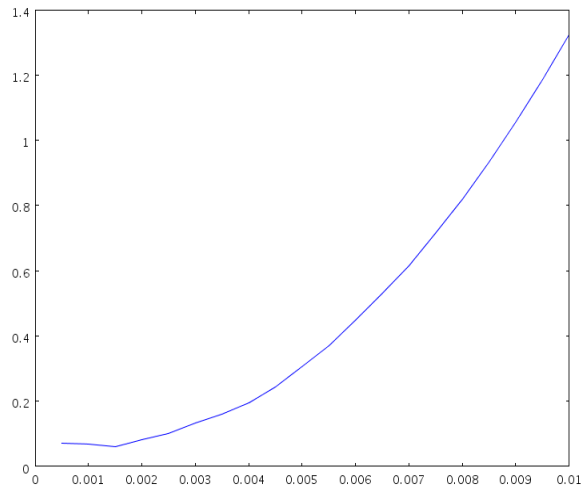}
\caption{The cost function (which slightly decreases before
increasing)}\label{ShapeVol}
\end{center}
\end{figure}
\begin{figure}[!h]
\begin{center}
\includegraphics[width=8.5cm,height=7cm]{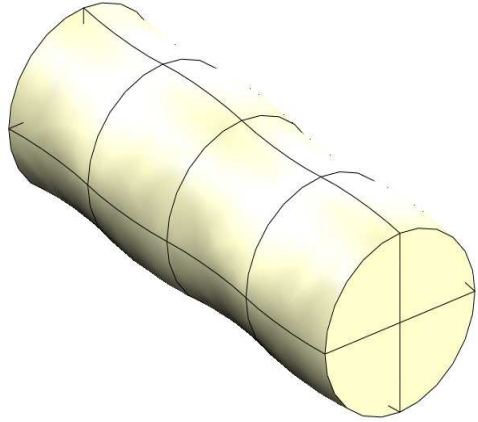}
\caption{A 3D (axially symmetric) shape which is better than the
cylinder}\label{ShapeVol2}
\end{center}
\end{figure}

\end{document}